\documentclass{amsart}

\usepackage{amsmath}
\usepackage{amsthm}
\usepackage{amssymb}
\usepackage{graphicx}

\newtheorem{theorem}{Theorem}

\newtheorem{corollary}[theorem]{Corollary}

\newtheorem{lemma}[theorem]{Lemma}
\newtheorem{proposition}[theorem]{Proposition}

\newtheorem{definition}[theorem]{Definition}

\usepackage{tikz}
\usetikzlibrary{calc}

\usetikzlibrary{quotes,angles}

\title{The fundamental inequality for cocompact Fuchsian groups}
\date{April 17, 2021}

\author{Petr Kosenko}
\address{University of Toronto, Canada}
\address{Higher School of Economics, Russia}
\email{petr.kosenko@mail.utoronto.ca}
\author{Giulio Tiozzo}
\address{University of Toronto, Canada}
\email{tiozzo@math.utoronto.ca}

\begin{document}

\begin{abstract}
We prove that the hitting measure is singular with respect to Lebesgue measure for any random walk on a cocompact Fuchsian group generated by translations joining opposite sides of a symmetric hyperbolic polygon. Moreover, the Hausdorff dimension of the hitting measure is strictly less than $1$. A similar statement is proven for Coxeter groups. 

Along the way, we prove for cocompact Fuchsian groups a purely geometric inequality for geodesic lengths, 
strongly reminiscent of the %with exactly the same form as the one of 
Anderson-Canary-Culler-Shalen inequality for free Kleinian groups.
\end{abstract}

\maketitle
Let $G < GL_2(\mathbb{R})$ be a countable group, and $\mu$ be a finitely supported, generating probability measure on $G$. 
We consider the random walk 
$$w_n := g_1 g_2 \dots g_n$$
where each $(g_i)$ is i.i.d. with distribution $\mu$. Let us fix a base point $o \in \mathbb{H}^2$. Then the \emph{hitting measure} $\nu$ of the random walk 
on $S^1 = \partial \mathbb{D}$ is 
$$\nu(A) := \mathbb{P}\left(\lim_{n \to \infty} w_n o \in A \right)$$
for any Borel set $A \subseteq \partial \mathbb{D}$. 
The hitting measure is also the unique $\mu$-\emph{harmonic}, or $\mu$-\emph{stationary}, measure, 
as it satisfies the convolution equation $\nu = \mu \star \nu$. 
On the other hand, the boundary circle $\partial \mathbb{D} = S^1$ also carries the \emph{Lebesgue} measure, which is the unique rotationally invariant measure on $S^1$.

\smallskip
In the 1970's, Furstenberg \cite{Fu} proved that for any discrete subgroup of $SL_2(\mathbb{R})$ there exists a measure $\mu$ such that the hitting measure of the corresponding random walk is absolutely continuous with respect to Lebesgue measure.
This was the first step to produce boundary maps, eventually leading to rigidity results. 
However, such measures $\mu$ are inherently infinitely supported, as they arise from discretization of Brownian motion (see also \cite{LS}). Another construction of absolutely continuous hitting measures, still infinitely supported, on general hyperbolic groups is given by \cite{Connell-Muchnik}.

On the other hand, if one replaces the random walk with a Brownian motion, then absolute continuity of harmonic measure only holds if the underlying manifold is highly homogeneous: to be precise, on a negatively curved surface, the hitting measure is absolutely continuous if and only if the curvature is constant (\cite{Le1}, \cite{Le2}). 

These two observations lead to the following

\smallskip
\textbf{Question A.} \emph{Given a finitely supported measure $\mu$ on $GL_2(\mathbb{R}) = \textup{Isom}(\mathbb{D})$, can the hitting measure for the random walk driven by $\mu$ be absolutely continuous with respect to Lebesgue measure?}
\smallskip

This question has been asked several times: in particular, Kaimanovich-Le Prince (\cite{KP}, page 259) conjectured that \emph{every} finitely supported measure on $SL_d(\mathbb{R})$ yields a singular hitting measure.

However, if one allows the group $G$ generated by the support of $\mu$ to be non-discrete, 
then there exist finitely supported measures which are absolutely continuous at infinity (\cite{Bo}, \cite{BPS}).

If $G$ is discrete but not cocompact, then the hitting measure is always singular: 
this fact has been proven in many contexts and with different proofs (\cite{GL}, \cite{BHM}, \cite{DKN}, \cite{KP}, \cite{GMT}, \cite{DG}, \cite{RT}), 
all of which exploit in various ways the fact that the cusp subgroup is highly distorted in $G$.

Thus, the only case which is still open is when $\mu$ is finitely supported and $G$ is discrete and cocompact.
Note that in this case the hyperbolic metric and the word metric are quasi-isometric to each other, hence distortion arguments 
do not work. 

%Any cocompact Fuchsian group has a fundamental domain given by a hyperbolic polygon $P$. 
By Poincar\'e's theorem (see e.g. \cite{Ma}), any cocompact Fuchsian group $G$ can be presented by identifying pairs of sides of a hyperbolic polygon. We consider a centrally symmetric polygon $P$, and take the group $G$ generated by hyperbolic translations identifying opposite sides. In order to apply Poincar\'e's theorem, we need the \emph{cycle condition} to hold: i.e., the sum of the angles at alternating vertices of $P$ must be of the form $\frac{2 \pi}{k}$ for some integer $k \geq 1$, see also Definition \ref{D:cycle}.

We prove the following. 
%satisfying the \emph{cycle condition}, i.e. the sum of the angles corresponding to each $G$-orbit of vertices of $P$ 
%In this paper, we consider a centrally symmetric polygon $P$, and take the group generated by identifying opposite sides. 
%The classical Poincar\'e's theorem states that , the sum of all angles must be of the form $\frac{2 \pi}{k}$ for $k \geq 1$ an integer. 
%Under the assumption that the fundamental domain for $G$ is centrally symmetric, we prove the following. 
%has a fundamental domain given by a hyperbolic polygon $P$. 
 % Cocompact Fuchsian groups are usually defined by identifying sides of a polygon
%By Poincar\'e's theorem (see e.g. \cite{Ma}), a way to define a cocompact Fuchsian group is to 
%identi whose interior angles satisfy the \emph{cycle condition} .
%We consider the group generated by translations identifying opposite sides
%, i.e. the sum of the angles corresponding to each $G$-orbit 
%of vertices of $P$ must be of the form $\frac{2 \pi}{k}$ for some integer $k \geq 1$ (Definition \ref{D:cycle}). 
%This is called the \emph{cycle condition}. 

%\begin{theorem} \label{T:main}
%Let $P$ be a centrally symmetric hyperbolic polygon in the Poincar\'e disk $\mathbb{D}$, with $2m$ sides, 
%satisfying the cycle condition, and let
%$T := \{ t_1, t_2, \dots, t_{2 m} \}$ 
%be the hyperbolic translations which identify opposite sides of $P$. 
%Then, for any measure $\mu$ supported on the set $T$, 
%the hitting measure on $S^1 = \partial \mathbb{D}$ is singular with respect to Lebesgue measure.
%\end{theorem}

\begin{theorem} \label{T:main}
Let $P$ be a centrally symmetric hyperbolic polygon in the Poincar\'e disk $\mathbb{D}$, with $2m$ sides, 
satisfying the cycle condition, and let
$T := \{ t_1, t_2, \dots, t_{2 m} \}$ 
be the hyperbolic translations which identify opposite sides of $P$. 
Then, for any measure $\mu$ supported on the set $T$, 
the hitting measure on $S^1 = \partial \mathbb{D}$ is singular with respect to Lebesgue measure.
\end{theorem}

%\begin{center}
%\includegraphics[width = 0.5 \textwidth]{Plot-Tess-4.pdf}
%\end{center}

\begin{center}
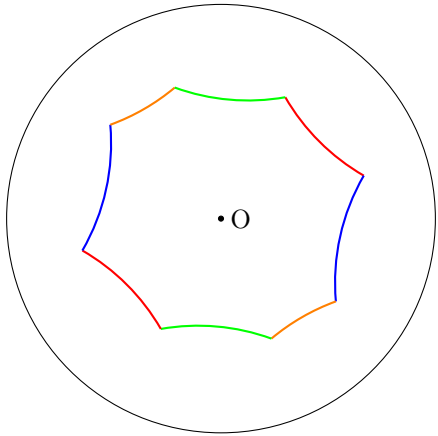
\begin{figure}[h!]
\begin{tikzpicture}[scale = 0.95]
\draw (0, 0) circle (3cm);
\draw [thick, color=red] (0.9, 1.7) arc (210 : 240 : 3 cm) node (a) {}; 
\draw [thick, color=blue] (a) arc (150 : 185 : 3 cm) node (b) {}; 
\draw [thick, color=orange] (b) arc (110 : 130 : 3 cm) node(c) {}; 
\draw [thick, color=green] (c) arc (70 : 100 : 3 cm) node (d) {};
\draw [thick, color = red] (d) arc (30 : 60 : 3 cm) node (e) {}; 
\draw [thick, color = blue] (e) arc (-30 : 5 : 3 cm) node (f) {}; 
\draw [thick, color = orange] (f) arc (-70 : -50 : 3 cm) node (g){}; 
\draw [thick, color = green ](g) arc (-110 : -80 : 3 cm) ;
\filldraw[black] (0,0) circle (1pt) node[anchor=west] {O};
\end{tikzpicture}
\caption{A symmetric hyperbolic octagon. Sides of the same color are identified by the Fuchsian group $G$.}
\end{figure}
\end{center}

Recently, singularity of hitting measure for cocompact Fuchsian groups whose fundamental domain is a \emph{regular} polygon has been proven by \cite{Kosenko} and \cite{CLP}, except for a finite number of cases with few sides.
(Note that the current paper is the first one to deal with the case of a hyperbolic octagon, even the regular one).

\medskip
We also prove the following version for reflection groups.

\begin{theorem} \label{T:main-cox}
Let $P$ be a centrally symmetric, hyperbolic polygon with $2m$ sides and interior angles $\frac{\pi}{k_i}$, 
with $k_i \in \mathbb{N}^+$ for $1 \leq i \leq 2m$.
Let $\mu$ be a probability measure supported on the set $R := \{ r_1, \dots, r_{2m} \}$ of hyperbolic reflections
on the sides of $P$, with $\mu(r_i) = \mu(r_{i+m})$ for all $1 \leq i \leq m$. Then the hitting measure
for the random walk driven by $\mu$ is singular with respect to Lebesgue measure. 
\end{theorem}

\medskip
\subsection*{The fundamental inequality} This problem is closely related to the following ``numerical characteristics" of random walks. 
Recall that the \emph{entropy} \cite{Avez} of $\mu$ is defined as
$$h := \lim_{n \to \infty} \frac{- \sum_{g \in G} \mu^n(g) \log \mu^n(g)}{n}$$
and the \emph{drift}, or \emph{rate of escape}, is 
$$\ell := \frac{d_\mathbb{H}(o, w_n o)}{n},$$
where $d_\mathbb{H}$ denotes the hyperbolic metric and the limit exists almost surely.
The drift also equals the classical Lyapunov exponent for random matrix products \cite{FK}.
Finally, the \emph{volume growth} of $G$ is 
$$v := \limsup_{n \to \infty} \frac{1}{n} \log \# \{ g \in G \ : \ d_\mathbb{H}(o, go) \leq n \}.$$
The inequality 
\begin{equation} \label{E:hlv}
h \leq \ell v
\end{equation}
has been established by Guivarc'h \cite{Guivarch} and is called the \emph{fundamental inequality} by Vershik \cite{Vershik}. 
Several authors (e.g. \cite[Question A]{Vershik}) have asked: 

\smallskip
\textbf{Question B.} \emph{Under which conditions is inequality \eqref{E:hlv} an equality?} 
\smallskip

For discrete, cocompact actions, Question B is equivalent to Question A: indeed, by \cite{BHM} (see also \cite{GMM} and Theorem \ref{T:BHM}), inequality \eqref{E:hlv} is strict if and only if the hitting 
measure is singular with respect to Lebesgue measure. 

If one replaces the hyperbolic metric $d_\mathbb{H}$ with a \emph{word metric} on $G$, then 
 \cite{GMM} prove that the inequality is strict unless the group $G$ is virtually free. 
%it is proven that, replacing $d_\mathbb{H}$ above with a \emph{word metric} on $G$, 
Observe that on cocompact Fuchsian groups any word metric is quasi-isometric to the hyperbolic metric; however, 
being quasi-isometric is not strong enough a condition to guarantee that the hitting measure is in the same class, 
hence the result from \cite{GMM} does not settle Question A. %the conjecture from \cite{KP}.
Note that for a cocompact Fuchsian group it is well-known that $v = 1$ (see e.g. \cite{PR}). 

Our result also has consequences on the Hausdorff dimension of the hitting measure. 
Recall that the Hausdorff dimension of a measure $\nu$ on a metric space is 
the infimum of the Hausdorff dimensions of subsets of full measure. 
%$$\textup{dim}_H(\nu) := \inf \{ \textup{dim}_H(A)  \subseteq X \ : \ \nu(A) = 1 \}.$$ 
Moreover, by \cite{Tan}, for cocompact Fuchsian groups, the Hausdorff dimension of the hitting measure 
satisfies, for almost every $x \in S^1$,
$$\textup{dim}_H(\nu) = \lim_{r \to 0^+} \frac{\log \nu(B(x, r))}{\log r} = \frac{h}{\ell}.$$

Thus, Theorem \ref{T:main} implies:  

\begin{corollary}
Under the hypotheses of Theorem \ref{T:main}, and of Theorem \ref{T:main-cox}, the inequality $h < \ell$ is strict.
Hence, the hitting measure $\nu$ has Hausdorff dimension strictly less than $1$.
\end{corollary} 

The approach of this paper is based on the fact that cocompactness forces at least some of the generators 
to have long enough translation lengths (this is related to the \emph{collar lemma}: two intersecting closed geodesics cannot be both short at the same time; also, the quotient Riemann surface has a definite positive area). Indeed, in Theorem \ref{T:criterion} we prove a criterion for singularity in terms of the translation lengths of the generators, and then we show the following purely geometric inequality.

\begin{theorem} \label{T:ineq-mcs-intro}
Let $P$ be a centrally symmetric polygon with $2m$ sides, satisfying the cycle condition, and let 
$S := \{ g_1, \dots, g_{2 m} \}$ be the set of hyperbolic translations identifying opposite sides of $P$. Then 
we have
\begin{equation}
\label{E:McS-intro}
\sum_{g \in S} \frac{1}{1 + e^{\ell(g)}} < 1,
\end{equation}
where $\ell(g)$ denotes the translation length of $g$ in the hyperbolic metric. 
\end{theorem}

Interestingly, our geometric inequality has exactly the same form as the one in \cite{CS}, \cite{ACCS}. However, 
it is not a consequence of theirs; see Section \ref{S:geometry}.

\subsection*{Acknowledgements}
This material is based upon work supported by the National Science Foundation under Grant No. DMS-1928930 while the second named author participated in a program hosted by the Mathematical Sciences Research Institute in Berkeley, California, during the Fall 2020 semester.
The second named author is also partially supported by NSERC and the Sloan Foundation. 
We thank R. Canary, I. Gekhtman, C. McMullen and G. Panti for useful comments on the first draft.

\section{Preliminary results}

Let $\mu$ be a probability measure on a countable group $G$. We assume that $\mu$ is \emph{generating}, i.e. the semigroup generated by 
the support of $\mu$ equals $G$.
We define the \emph{step space} as $(G^\mathbb{N}, \mu^\mathbb{N})$, and the map $\pi : G^\mathbb{N} \to G^\mathbb{N}$ as $\pi ((g_n)_{n \in \mathbb{N}}) := (w_n)_{n \in \mathbb{N}}$, with for any $n$
$$w_n := g_1 g_2 \dots g_n.$$
The target space of $\pi$ is denoted by $\Omega$ and called the \emph{path space}; as a set, it equals $G^\mathbb{N}$ , and is equipped with the measure $\mathbb{P}_\mu := \pi_\star (\mu^\mathbb{N})$.

\medskip
Then, we define the \emph{first-passage function} $F_\mu(x, y)$ as 
$$F_\mu(x, y) := \mathbb{P}_\mu(\exists n \ : \ w_n x = y)$$
for any $x, y \in G$, and the  \emph{Green metric} $d_\mu$ on $G$, introduced in \cite{BB}, as 
$$d_\mu(x, y) := - \log F_\mu(x, y).$$
The following fact is well-known. 

\begin{lemma} \label{L:homo}
Let $p : G \to H$ be a group homomorphism, let $\mu$ be a probability measure on $G$, and let $\overline{\mu} := p_\star \mu$. Then, 
for any $x, y \in G$,
$$d_{\overline{\mu}}(p(x), p(y)) \leq d_\mu(x, y).$$
\end{lemma}

\begin{proof}
Since $p$ induces a map from paths in $G$ to paths in $H$, we have $\overline{\mu}^n(p(g)) \geq \mu^n(g)$ for any $g \in G$, any $n \geq 0$. 
Hence 
$$\mathbb{P}_{\overline{\mu}}(p(x), p(y)) \geq \mathbb{P}_\mu(x, y)$$
for any $x, y \in G$, from which the claim follows.
\end{proof}

We shall use the following criterion, which relates the absolute continuity of the hitting measure to the fundamental inequality.
Recall that a group action is \emph{geometric} if it is isometric, properly discontinuous, and cocompact.

\begin{theorem}\textup{(\cite[Corollary 1.4, Theorem 1.5]{BHM},  \cite{Tan}, \cite{GT})} \label{T:BHM}
Let $\Gamma$ be a non-elementary hyperbolic group acting geometrically on $\mathbb{H}^2$, 
endowed with the geometric distance $d = d_\mathbb{H}$ induced
from the action. Consider a generating probability measure $\mu$ on $\Gamma$ with finite support. 
%Let us also assume that $\mu$ is symmetric. 
Then the following conditions are equivalent:
\begin{enumerate}
\item
The equality $h = \ell v$ holds.
\item
The Hausdorff dimension of the hitting measure $\nu$ on $S^1$ is equal to $1$.
\item
The measure $\nu$ is equivalent to the Lebesgue measure on $S^1$.
\item
For any $o \in \mathbb{H}^2$, there exists a constant $C > 0$ such that for any $g \in \Gamma$ we have
$$|d_\mu(e, g) - d_\mathbb{H}(o, g o)| \leq C.$$
\end{enumerate}
\end{theorem}

\medskip

For each $g \in G$, let $\ell(g)$ denote its translation length, namely 
$$\ell(g) := \lim_{n \to \infty} \frac{d_{\mathbb{H}}(o, g^n o)}{n}.$$
Equivalently, $\ell(g)$ is the length of the corresponding closed geodesic on the quotient surface.
The mechanism to utilize Theorem \ref{T:BHM} is through the following lemma, similar to the one from \cite{Kosenko}.

\begin{lemma} \label{L:small-tr}
Suppose that the hitting measure is absolutely continuous. Then for any $g \in G$ we have 
$$\ell(g) \leq d_\mu(e, g).$$
\end{lemma}

\begin{proof}
If not, then $\ell(g) > d_\mu(e, g) \geq 0$, hence $g$ is loxodromic. Let us pick some $o \in \mathbb{H}^2$ which lies on the axis of $g$, so that 
$d_\mathbb{H}(o, g^k o) = \ell(g^k) = k \ell(g)$ for any $k$. Moreover, by the triangle inequality for the Green metric one has $d_\mu(e, g^k) \leq k d_\mu(e, g)$, hence 
$$d_\mathbb{H}(o, g^k o) - d_\mu(e, g^k) \geq k \ell(g) - k d_\mu(e, g) = k (\ell(g) - d_\mu(e, g))$$
thus, since  $\ell(g) - d_\mu(e, g) > 0$, 
$$\sup_{k \in \mathbb{N}} \left|d_\mathbb{H}(o, g^k o) - d_\mu(e, g^k) \right| = + \infty,$$
which contradicts Theorem \ref{T:BHM}. 
\end{proof}

Let $F$ be a free group, freely generated by a finite set $S$. Recall the (hyperbolic) \emph{boundary} $\partial F$ of $F$ is the set of 
infinite, reduced words in the alphabet $S \cup S^{-1}$. Given a finite, reduced word $g$, 
we denote as $C(g) \subseteq \partial F$ the \emph{cylinder} determined by $g$, namely 
the set of infinite, reduced words which start with $g$. 

\begin{lemma} \label{L:cylinder}
	Consider a random walk on the free group 
	$$F_m = \left\langle s_1^{\pm 1}, \dots, s_m^{\pm 1} \right\rangle,$$ 
	defined by a probability measure $\mu$ on the generators. If we denote $x_i := F_\mu(e, s_i)$, $\check{x}_i := F_\mu(e, s_i^{-1})$, and the hitting measure on the boundary of $F_m$ by $\nu$, then 
	\[
	\nu(C(s_i)) = \frac{x_i (1 - \check{x}_{i})}{1 - x_i \check{x}_{i} }.
	\]
\end{lemma}

A similar lemma is stated in \cite[Exercise 5.14]{Lalley}.

\begin{proof}
	For any infinite word $w = s_{j_1} s_{j_2} s_{j_3} \dots$ there exist two possibilities:
	\begin{enumerate}
		\item There exists a subword $s_{j_1} \dots s_{j_k}$ such that it equals $s_i$ in $F_m$
		\item No subword $s_{j_1} \dots s_{j_k}$ equals $s_i$, so it belongs to the set of paths which never hit $s_i$.
	\end{enumerate}
	In the first case we denote this subword by $w_1$, and we consider $w_1^{-1} w$ and we apply the same procedure, but replacing $s_i$ with $s_i^{-1}$ at each subsequent step. This procedure yields the equality
	\[
	\begin{aligned}
		\nu(C(s_i)) &= \mathbb{P}(e \rightarrow s_i \nrightarrow e) + \mathbb{P}(e \rightarrow s_i \rightarrow e \rightarrow s_i \nrightarrow e) + \dots = \\ 
		&= \sum_{n=0}^{\infty} F_\mu(e, s_i)^{n+1} F_\mu(e, s_i^{-1})^{n} (1 - F_\mu(e, s_i^{-1})) \\
		& = F_\mu(e, s_i)(1 - F_\mu(e, s_i^{-1})) \sum_{n=0}^{\infty} \left(F_\mu(e, s_i)F_\mu(e, s_i^{-1})\right)^n = \frac{x_i (1 - \check{x}_{i})}{1 - x_i \check{x}_{i} }.
	\end{aligned}
	\]
\end{proof}

\section{A criterion for singularity}

\begin{theorem} \label{T:criterion}
Let $\mu$ be a finitely supported measure on a cocompact Fuchsian group, and let $S$ be the support of $\mu$. 
Suppose that 
\begin{equation}
\label{E:McS2}
\sum_{g \in S} \frac{1}{1 + e^{\ell(g)}} < 1.
\end{equation}
Then the hitting measure $\nu$ on $\partial \mathbb{D}$ is singular with respect to Lebesgue measure. 
\end{theorem}

\medskip

\begin{proof}
Let $F$ be a free group of rank $m$, with generators $(h_i)_{i =1}^{m}$, and let 
$\widetilde{\mu}$ be a measure on $F$ with $\widetilde{\mu}(h_i^\pm) = \mu(g_i^\pm)$. 
Moreover, let us denote
\begin{align*}
x_i &:= F_{\widetilde{\mu}}(e, h_i) =  \mathbb{P}_{\widetilde{\mu}}( \exists n \ : \ w_n = h_i) \\ 
\check{x}_i & := F_{\widetilde{\mu}}(e, h_i^{-1}).
\end{align*}
Then we have 
\begin{equation} \label{E:all-measure}
\sum_{i =1}^{m}  \frac{x_i (1 - \check{x}_{i})}{1 - x_i \check{x}_{i} } + \frac{\check{x}_i (1 - x_{i})}{1 - x_i \check{x}_{i} } = 1.
\end{equation}
Indeed, if $\widetilde{\nu}$ is the hitting measure on $\partial F$, by Lemma \ref{L:cylinder} the measure of the cylinder $C(h_i)$ starting with $h_i$ is
$$\widetilde{\nu}(C(h_i)) = \frac{x_i (1 - \check{x}_{i})}{1 - x_i \check{x}_{i} }, \qquad \widetilde{\nu}(C(h_i^{-1})) = \frac{\check{x}_i (1 - x_{i})}{1 - x_i \check{x}_{i} }$$
from which, since the cylinders are disjoint and cover the boundary, \eqref{E:all-measure} follows.

Then, by equation \eqref{E:McS2}, there exists an index $i$ such that 
$$ \frac{2}{1 + e^{\ell(g_i)}} < \frac{x_i (1 - \check{x}_{i})}{1 - x_i \check{x}_{i} } + \frac{\check{x}_i (1 - x_{i})}{1 - x_i \check{x}_{i} }$$
which is equivalent to 
$$e^{\ell(g_i)} > \frac{2 - x_i - \check{x}_i}{x_i + \check{x}_i - 2 x_i \check{x}_i}$$
Finally, an algebraic computation yields
$$\frac{2 - x_i - \check{x}_i}{x_i + \check{x}_i - 2 x_i \check{x}_i} \geq \min \left\{ \frac{1}{x_i}, \frac{1}{\check{x}_i} \right\}$$
thus we obtain 
\begin{equation} \label{E:exp-l}
\ell(g_i) > \inf \{ - \log x_i, - \log \check{x}_i \}.
\end{equation}

If the hitting measure $\nu$ on $S^1 = \partial \mathbb{D}$ is absolutely continuous, then by Lemma \ref{L:small-tr}
and Lemma \ref{L:homo} we get
$$\ell(g_i) \leq d_{\mu}(e, g_i) \leq d_{\widetilde{\mu}}(e, h_i) = - \log x_i$$ 
for any $i$. If we apply the same inequality to $g_i^{-1}$, we also have 
$$\ell(g_i)  = \ell(g_i^{-1}) \leq d_{\mu}(e, g_i^{-1}) \leq d_{\widetilde{\mu}}(e, h_i^{-1}) = - \log \check{x}_i$$ 
hence 
$$\ell(g_i) \leq \inf  \{ -\log x_i, -\log \check{x}_i \}$$
which contradicts \eqref{E:exp-l}, showing that $\nu$ is singular with respect to Lebesgue measure.
\end{proof}

%Hence 
%$$\frac{x_i}{1 + x_i} \leq \frac{1}{1 + e^{\ell(g_i)}}$$
%and, by summing over all $i$'s,
%$$1 = 2 \sum_{i=1}^{m} \frac{x_i}{1 + x_i} \leq 2 \sum_{i=1}^{m} \frac{1}{1 + e^{\ell(g_i)}} <1,$$
%which contradicts \eqref{E:McS2}. Hence, 
%Hence, since the cylinders are disjoint and cover the boundary, we have 
%$$\sum_{i =1}^{m}  \frac{x_i (1 - \check{x}_{i})}{1 - x_i \check{x}_{i} } + \frac{\check{x}_i (1 - x_{i})}{1 - x_i \check{x}_{i} } = 1.$$

\section{Parameterization of the space of polygons}  \label{S:geometry}

%Following \cite{Ga}, a \emph{hyperelliptic polygon} is a polygon $P$ in the hyperbolic plane, with $2m$ 
%sides $(L_i)_{i = 1}^{2m}$ which are geodesic for  the hyperbolic metric, and such that: 
%\begin{enumerate}
%\item $P$ is invariant under the rotation $j(z) = -z$; hence, opposite sides have equal length;
%\item The sum of all angles of $P$ is $2 \pi$. 
%\end{enumerate}

%Given a centrally symmetric polygon $P$, one considers for each pair $(L_i, L_{i+m})$ of opposite sides a hyperbolic translation $g_i$ 
%such that $g_i(L_i) = L_{i+m}$. By \cite{Ga}, the set $(g_i)_{i=1}^m$ generates a cocompact, hyperelliptic Fuchsian group, 
%and moreover, every hyperelliptic Fuchsian group has a fundamental domain which is a hyperelliptic polygon $P$ constructed as above. 

Let $P$ be a convex, compact polygon in the hyperbolic disk $\mathbb{D}$, with $2m$ sides and interior angles $\{ \gamma_1, \dots, \gamma_{2m} \}$. 

We say that $P$ is \emph{centrally symmetric} if there exists a point $o \in \mathbb{D}$ so that $P$ is invariant under reflection across $o$. This clearly implies that opposite sides have equal length, and opposite angles are equal. 

Poincar\'e's theorem provides conditions to ensure that the group generated by side pairings is discrete (see \cite{Ma}). 
% shows that, in order to generate a discrete subgroup of $\textup{Isom}(\mathbb{D})$, 
In particular, one needs a condition on the angles, which in our setting can be formulated as follows. 
%Given a polygon $P$ and a group $G$ acting on $\mathbb{D}$, the vertices of $P$ are divided in equivalence classes, called \emph{elliptic cycles}, where two vertices are identified if they lie in the same $G$-orbit. Poincar\'e's cycle condition states that, for each elliptic cycle, the sum of the angles at the vertices of this cycle must be of the form $\frac{2 \pi}{k}$ for some integer $k  > 1$. 
%In our case, we identify opposite sides, hence there are two possibilities: if $m$ is even, all vertices are identified; if $m$ is odd, 
%there are two elliptic cycles, corresponding to alternate vertices of $P$. Thus, the cycle condition can be formulated as follows.

\begin{definition}\label{D:cycle}
A centrally symmetric polygon $P$ satisfies the \emph{cycle condition} if there exists an integer $k \geq 1$ such that 
$$\sum_{i = 1}^{m} \gamma_{2 i} = \sum_{i = 1}^m \gamma_{2i - 1} = \frac{2 \pi}{k}.$$
\end{definition}

Let $S := \{ g_1, \dots, g_{2 m} \}$ be the set of hyperbolic translations identifying opposite sides of $P$.
By Poincar\'e's theorem \cite{Ma}, if the polygon $P$ satisfies the cycle condition, then the group $G$ generated by $S$ is discrete\footnote{Note that in the usual formulation of Poincar\'e's theorem there are two cases:  if $m$ is even, all vertices of $P$ are identified by $G$; if $m$ is odd, there are two elliptic cycles, corresponding to alternate vertices of $P$. If $m$ is even and $k = 1$, the polygon $P$ does not satisfy the classical version of Poincar\'e's theorem, but if $P$ is symmetric, the group generated is still discrete, so all our arguments still apply.}.

\medskip
The following is our main geometric inequality. 

\begin{theorem} \label{T:ineq-mcs}
Let $P$ be a centrally symmetric, hyperbolic polygon satisfying the cycle condition, with $2m$ sides, and let $S := \{ g_1, \dots, g_{2 m} \}$ be the set of hyperbolic translations identifying opposite sides of $P$.
%and let $G$ be its associated cocompact Fuchsian group. Let . 
Then we have
\begin{equation}
\label{E:McS3}
\sum_{g \in S} \frac{1}{1 + e^{\ell(g)}} < 1.
\end{equation}
\end{theorem}

\begin{center}
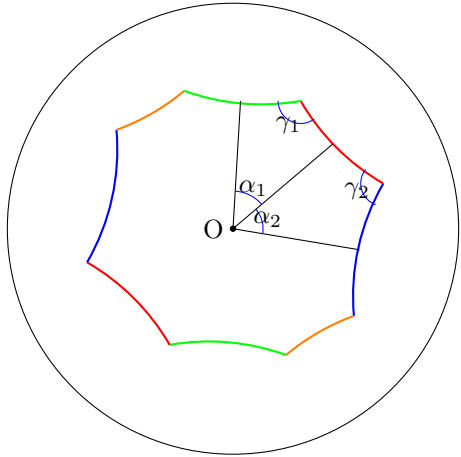
\begin{figure}[h!]
\begin{tikzpicture}
\draw (0, 0) circle (3cm) node (o){};
\draw [thick, color=red] (0.9, 1.7) arc (210 : 240 : 3 cm) node (a) {}; 
\draw [thick, color=blue] (a) arc (150 : 185 : 3 cm) node (b) {}; 
\draw [thick, color=orange] (b) arc (110 : 130 : 3 cm) node(c) {}; 
\draw [thick, color=green] (c) arc (70 : 100 : 3 cm) node (d) {};
\draw [thick, color = red] (d) arc (30 : 60 : 3 cm) node (e) {}; 
\draw [thick, color = blue] (e) arc (-30 : 5 : 3 cm) node (f) {}; 
\draw [thick, color = orange] (f) arc (-70 : -50 : 3 cm) node (g){}; 
\draw [thick, color = green ](g) arc (-110 : -80 : 3 cm) node (h){} ;
\filldraw[black] (0,0) circle (1pt) node[left] {O};
\draw (0,0) -- (0.1, 1.7) node (x){}; 
\draw (0,0) -- (1.33, 1.13) node (y){}; 
\draw (0,0) -- (1.68, -0.28) node (z){}; 

\draw pic["$\gamma_1$", draw=blue, -, angle eccentricity=1.2, angle radius=0.3 cm]
    {angle=x--h--y};
    \draw pic["$\gamma_2$", draw=blue, -, angle eccentricity=1.2, angle radius=0.3 cm]
    {angle=y--a--z};

\draw pic["$\alpha_1$", draw=blue, -, angle eccentricity=1.2, angle radius=0.5 cm]
    {angle=y--o--x};
\draw pic["$\alpha_2$", draw=blue, -, angle eccentricity=1.2, angle radius=0.4 cm]
    {angle=z--o--y};
    
\end{tikzpicture}
\caption{Angles at the center and at the vertices of a symmetric hyperbolic octagon.}
\end{figure}
\end{center}

\noindent \textbf{Remarks.} 

%Note that the inequality compares directly to \eqref{E:McS2} if one sums over both the translations and their inverses, so that the right-hand side becomes $1$.

%Note that if one replaces (2) above by the sum of the angles of $P$ being equal to $\frac{2 \pi}{k}$, with $k > 1$ an integer, Theorem \ref{T:ineq-mcs} (hence also Theorem \ref{T:main}) still holds, with the same proof. By Poincar\'e's theorem, this still yields a cocompact (though not torsion-free) Fuchsian group.

The inequality \eqref{E:McS3} has the same form as the main inequality in \cite{ACCS} and \cite{CS} for free Kleinian groups; 
 more recently, a stronger version for free Fuchsian groups has been obtained in \cite{He}, while generalizations in variable curvature (and any dimension) are due to \cite{Ho}, \cite{BM}.

Equation \eqref{E:McS3} is also reminiscent of McShane's identity \cite{McShane}, where one obtains the equality by taking the infinite sum over all group elements
of a punctured torus group. Our inequality, however, does not follow from any of them; in fact, it is in a way stronger than these, as a cocompact surface group can be deformed to a finite covolume group and then to a Schottky (hence free) group by increasing the translation lengths of the generators.

It is interesting to point out that the above inequalities have an interpretation in terms of hitting measures of stochastic processes (see e.g. \cite{LT}). 
Here, we go along the opposite route: we prove the geometric inequality \eqref{E:McS3} and then we use it to conclude properties 
about the hitting measure.

Finally, there are generating sets of $G$ for which \eqref{E:McS3} fails. Indeed, the mechanism behind the inequality is that, 
since all curves corresponding to $(g_i)_{i = 1}^{m}$ intersect each other, by the collar lemma,  at most one of them can be short. 
In general, on a surface of genus $g$ one can choose a configuration of $3 g - 3$ short curves, and construct a Dirichlet domain
for which the corresponding side pairing does not satisfy \eqref{E:McS3}.

\medskip
\noindent \textbf{Proof.} 
The proof of this inequality will take up most of the paper, until Section \ref{S:obtuse}.
To begin with, let us note that a way to parameterize the space of all symmetric hyperbolic polygons is to write, by \cite[Example 2.2.7]{Buser},
\begin{equation} \label{E:cosh}
\cos(\gamma_i) = - \cosh(a_i) \cosh(a_{i+1}) \cos(\alpha_i) + \sinh(a_i) \sinh(a_{i+1})
\end{equation}
with $i = 1, \dots, m$, where $(a_i)$ are the distances between the base point and the $i$th side, $(\alpha_i)$ are the angles at the origin and $(\gamma_i)$ are the angles at the vertices. 
Since $\ell(g_i) \geq 2 a_i$, it is enough to show
$$\sum_{i = 1}^m \frac{1}{1 + e^{2 a_i}} < \frac{1}{2}$$
under the constraints $\sum_{i=1}^m \alpha_i = \pi$ and $\sum_{i=1}^m \gamma_i = \pi$.

\medskip
The fundamental geometric idea in our approach to Theorem \ref{T:ineq-mcs} is that two intersecting curves cannot be both short, as a consequence of the 
\emph{collar lemma} \cite{Buser2}. For instance, we get: 

\begin{lemma}
Suppose that there exists $a_i$ such that $\sinh(a_i) \leq \frac{2 (m-1)}{m (m-2) }$.
Then the hitting measure is singular.
\end{lemma}

\begin{proof}
From the collar lemma \cite{Buser2} we have 
$$\sinh(a_i) \sinh(a_j) \geq 1$$
for all $i \neq j$. 
Recall that 
$$\frac{2}{1 + e^{2a}} = 1 - \tanh(a)$$
hence, if we set $s := \sinh(a_1)$, we obtain for $i  \neq 1$ that $\sinh(a_i) \geq \frac{1}{s}$  thus 
$$\tanh(a_i) = \frac{\sinh(a_i)}{\sqrt{1 + \sinh(a_i)^2} } \geq \frac{1}{\sqrt{1 +  s^2}}$$
hence
$$\sum_{i = 1}^m \tanh(a_i) \geq \frac{s}{\sqrt{1 + s^2}} + \frac{m-1}{\sqrt{1 + s^2}} > m -1 $$
if and only if $s < \frac{2 (m-1)}{m (m-2) }$.
\end{proof}

\medskip
To actually prove Theorem \ref{T:ineq-mcs}, however, we need an improvement on the previous estimate. 
Let us rewrite equation \eqref{E:cosh} above as 
$$\cos(\alpha_i)  = \tanh(a_i) \tanh(a_{i+1}) - \frac{\cos(\gamma_i)}{\cosh(a_i) \cosh(a_{i+1})}$$
and, recalling that
$$\tanh^2(x) + \frac{1}{\cosh^2(x)} = 1$$
we obtain, by setting $z_i = \tanh(a_i)$, 
\begin{equation} \label{E:z}
\cos(\alpha_i)  = z_i z_{i+1} - \cos(\gamma_i) \sqrt{1 - z_i^2} \sqrt{1 - z_{i+1}^2}
\end{equation}
with $0 \leq z_i \leq 1$. Finally, we want to show
$$\sum_{i = 1}^m \frac{1}{1 + e^{2 a_i}} = \sum_{i = 1}^m \frac{1 - z_i}{2} \overset{?}{<} \frac{1}{2},$$
which is equivalent to 
\begin{equation} \label{E:sum}
\sum_{i = 1}^m z_i \overset{?}{>} m - 1.
\end{equation}
Now, let us first assume that $\gamma_i \leq \pi/2$ for all $1 \leq i \leq m$. 
Then \eqref{E:z} yields
$$\cos(\alpha_i)  \leq z_i z_{i+1}$$
hence the constraint becomes 
\begin{equation} \label{E:arccos}
\sum_{i = 1}^m \arccos(z_i z_{i+1}) \leq \pi.
\end{equation}
Note that $z_1 \to 0$ implies $\cos \alpha_1 \leq z_1 z_2 \to 0$ thus $\alpha_1 \to \frac{\pi}{2}$ and $\cos \alpha_m \leq z_m z_1 \to 0$ thus $\alpha_m \to \frac{\pi}{2}$, hence also $\alpha_2, \alpha_3, \dots, \alpha_{m-1} \to 0$, which implies $z_2, z_3, \dots, z_m \to 1$.

\section{An optimization problem}

By the above discussion, we reduced the proof of Theorem \ref{T:ineq-mcs} (at least in the case all angles of $P$ are acute)
to the following optimization problem. 

\begin{theorem} \label{T:ineq}
Let $m \geq 3$ and $0 \leq x_i \leq 1$ with $\sum_{i=1}^m x_i = m - 1$.
Then 
$$\sum_{i = 1}^m  \arccos(x_i x_{i+1}) \geq \pi.$$
Moreover, equality holds if and only if there exists an index $i$ such that $x_i = 0$ and $x_j = 1$ for all $j \neq i$.
\end{theorem}

\begin{center}
\begin{figure}
\includegraphics[width = 0.6 \textwidth]{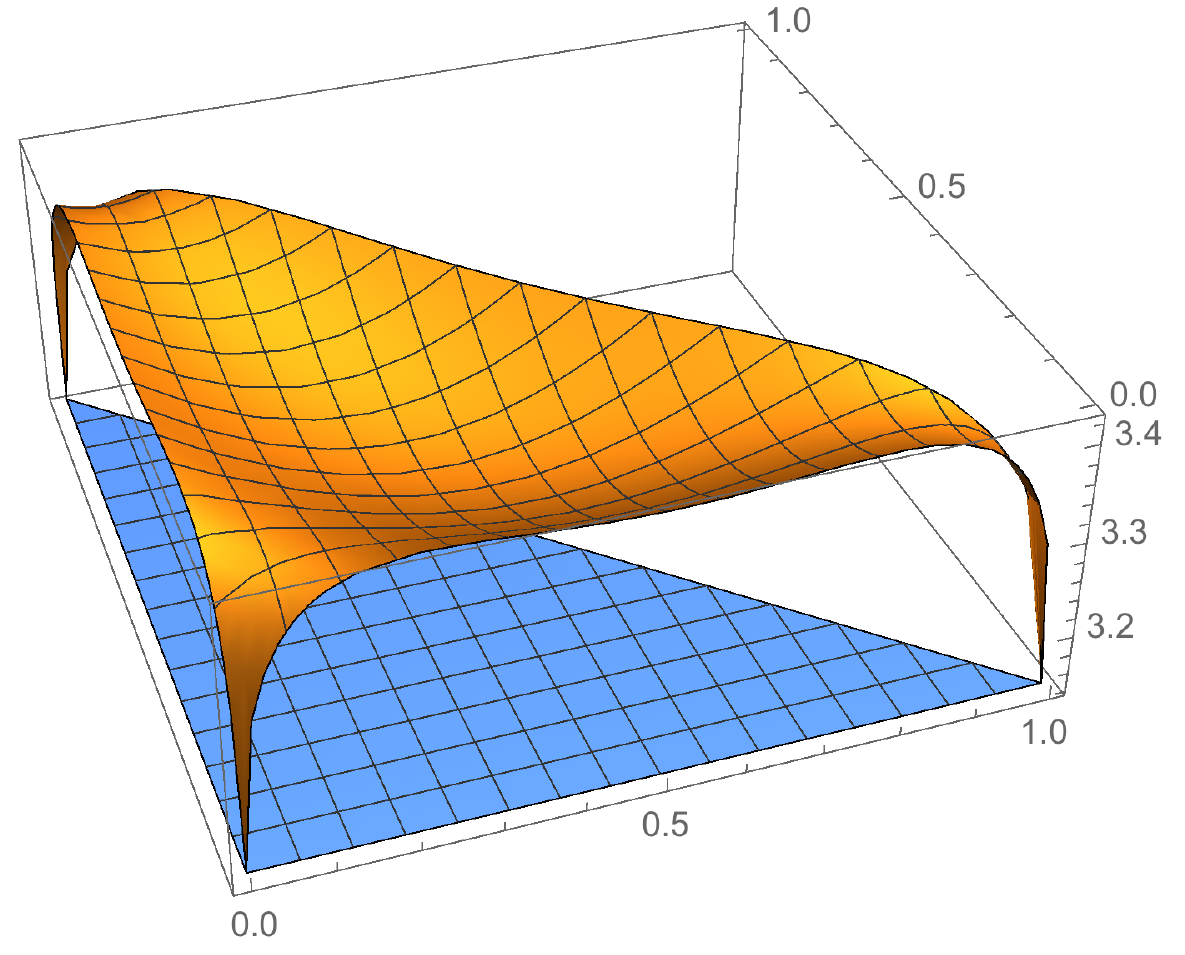}
\caption{The graph of $f(x) := \sum_{i = 1}^3 \arccos((1-x_i)(1-x_{i+1}))$ subject to the constraint $\sum_{i = 1}^3 x_i = 1$, compared 
with the constant function at height $\pi$.The lack of convexity (or concavity) of $f$ makes the proof of Theorem \ref{T:ineq} trickier.
}
\end{figure}
\end{center}

In the statement of Theorem \ref{T:ineq} and elsewhere from now on, all indices $i$ are meant modulo $m$.
The next is the main technical lemma. 

\begin{lemma} \label{T:sqrt}
Let $m \geq 3$ and $0 \leq x_i \leq 1$ with $\sum_{i=1}^m x_i = 1$.
Then
$$\sum_{i = 1}^m \sqrt{x_i + x_{i+1} - x_i x_{i+1}}  \geq  \sqrt{ 4 + 3 \sum_{i =1}^m x_i x_{i+1}}.$$
\end{lemma}

\begin{proof}
Set $\Delta_i := x_i + x_{i+1} - x_i x_{i+1}$.
Note that 
$$\Delta_i \geq \max\{ x_i, x_{i+1}\}$$
hence 
\begin{equation} \label{E:delta-bound}
\sqrt{\Delta_i}\sqrt{\Delta_{i+1}} \geq x_{i+1}.
\end{equation}
Moreover, since $m \geq 2$, we have $x_{i+1} + x_{i+2} \leq \sum_{i =1}^m x_i = 1$, hence if we multiply by $(x_{i+1} + x_{i+2})$, we obtain 
\begin{align*}
\Delta_i & = x_i + x_{i+1} - x_i x_{i+1} \\
& \geq  (x_i + x_{i+1}) (x_{i+1} + x_{i+2}) - x_i x_{i+1} \\
& \geq x_{i+1}^2 + x_{i+1} x_{i+2}.
\end{align*}
Similary, we obtain 
\begin{align*}
\Delta_{i+2}  & = x_{i+2} + x_{i+3} - x_{i+2} x_{i+3} \\
&  \geq (x_{i+2} + x_{i+3})(x_{i+1} + x_{i+2}) - x_{i+2} x_{i+3} \\
& \geq x_{i+2}^2 + x_{i+1} x_{i+2}.
\end{align*}
Thus, Cauchy-Schwarz yields
\begin{equation} \label{E:CS}
\sqrt{\Delta_i}\sqrt{\Delta_{i+2}} \geq \sqrt{ x_{i+1}^2 + x_{i+1} x_{i+2}} \sqrt{ x_{i+2}^2 + x_{i+1} x_{i+2}} \geq 2 x_{i+1} x_{i+2}.
\end{equation}
By squaring both sides, our desired inequality is equivalent to 
$$\sum_{i =1}^m \Delta_i + 2 \sum_{1 \leq i < j \leq m} \sqrt{\Delta_i} \sqrt{\Delta_j} \geq 4 + 3 \sum_{i =1}^m x_i x_{i+1},$$
thus, using $\sum_{i =1}^m \Delta_i = 2 - \sum_{i = 1}^m x_i x_{i+1}$, it is enough to prove 
\begin{equation} \label{E:equiv}
\sum_{1 \leq i < j \leq m} \sqrt{\Delta_i} \sqrt{\Delta_j} \geq 1 + 2 \sum_{i = 1}^m x_i x_{i+1}.
\end{equation}
Now, note that 
$$\sum_{1 \leq i < j \leq m} \sqrt{\Delta_i} \sqrt{\Delta_j}  = \sum_{i = 1}^m \sqrt{\Delta_i} \sqrt{\Delta_{i+1}} + M$$
with 
\begin{align}
M & = 0 & \textup{if } m = 3\phantom{.} \\
M & = \sum_{i = 1}^2 \sqrt{\Delta_i} \sqrt{\Delta_{i+2}} & \textup{if }m = 4\phantom{.} \\
\label{E:last}
M & \geq \sum_{i=1}^m \sqrt{\Delta_i} \sqrt{\Delta_{i+2}} & \textup{if }m \geq 5.
\end{align}
Thus, for $m \geq 5$ we have, using \eqref{E:last}, \eqref{E:delta-bound} and \eqref{E:CS}, 
\begin{align*}
\sum_{1 \leq i < j \leq m} \sqrt{\Delta_i} \sqrt{\Delta_j} & \geq
\sum_{i=1}^m \sqrt{\Delta_i} \sqrt{\Delta_{i+1}} +  \sum_{i=1}^m \sqrt{\Delta_i} \sqrt{\Delta_{i+2}} \\
& 
\geq 
\sum_{i=1}^m x_{i+1} +  2 \sum_{i=1}^m x_{i+1} x_{i+2} \\
& \geq 1 +  2\sum_{i=1}^m x_{i+1} x_{i+2}
\end{align*}
which yields \eqref{E:equiv}, hence completes our proof.
The cases $m = 3$ and $m = 4$ need to be dealt with separately.
If $m = 3$, we obtain, by multiplying by $\sum_{i = 1}^3 x_i = 1$, 
$$\Delta_i = x_i^2 + x_{i+1}^2 + \sum_{i = 1}^3  x_i x_{i+1}$$
so by Cauchy-Schwarz we get 
$$\sqrt{\Delta_i} \sqrt{\Delta_{i+1}} \geq x_{i+1}^2 + x_i x_{i+2} + \sum_{i = 1}^3 x_i x_{i+1}$$
hence 
\begin{align*}
\sum_{i=1}^3 \sqrt{\Delta_i} \sqrt{\Delta_{i+1}}   & \geq \sum_{i=1}^3 x_i^2 + 4 \sum_{i=1}^3 x_i x_{i+1} \\
&  = \left( \sum_{i=1}^3 x_i \right)^2 + 2 \sum_{i = 1}^3 x_i x_{i+1} \\
& = 1 + 2 \sum_{i = 1}^3 x_i x_{i+1}
\end{align*}
which yields \eqref{E:equiv}, as desired. 
Finally, if $m = 4$, then we note 
$$\sum_{1 \leq i < j \leq 4} \sqrt{ \Delta_i } \sqrt{\Delta_j} = \sum_{i = 1}^4 \sqrt{\Delta_i} \sqrt{\Delta_{i+1}} + \sum_{i = 1}^2 \sqrt{\Delta_i} \sqrt{\Delta_{i+2}}$$
and, again by Cauchy-Schwarz, 
$$\sqrt{\Delta_1} \sqrt{\Delta_3} \geq \sqrt{x_1^ 2 + x_2^2 + x_1 x_4 + x_2 x_3} \sqrt{x_3^ 2 + x_4^2 + x_1 x_4 + x_2 x_3}  \geq 2 x_1 x_4 + 2 x_2 x_3$$
and similarly 
$$\sqrt{\Delta_2} \sqrt{\Delta_4} \geq  2 x_1 x_2 + 2 x_3 x_4$$
thus, using \eqref{E:delta-bound}, 
$$\sum_{1 \leq i < j \leq 4} \sqrt{ \Delta_i } \sqrt{\Delta_j} \geq \sum_{i = 1}^4 x_i + 2 \sum_{i = 1}^4 x_i x_{i+1}  = 1  + 2 \sum_{i = 1}^4 x_i x_{i+1}$$
which is again \eqref{E:equiv}. This completes the proof.
\end{proof}

\begin{lemma} \label{T:calc}
For $0 \leq x \leq 1$ we have the inequalities:
\begin{enumerate}
\item
$$\frac{2}{\pi} \arccos(1-x) \geq \frac{2}{3} \sqrt{x} + \frac{1}{3} x$$
with equality if and only if $x = 0$ or $x= 1$;
\item
$$\frac{2}{3} \sqrt{4 + 3 x} + \frac{2-x}{3}  \geq 2 $$
with equality if and only if $x = 0$.
\end{enumerate}
\end{lemma}

\begin{proof}
For the first inequality, let $f(x) := \frac{2}{\pi} \arccos(1-x^2) - \frac{2}{3} x - \frac{1}{3} x^2$.  One checks that $f(0) = f(1) = 0$ and $f(\frac{1}{\sqrt{2}}) = \frac{1}{2} - \frac{\sqrt{2}}{3} > 0$; moreover, $f'(x)$ has a unique zero in $[0, 1]$. Hence, $f(x) \geq 0$ for all $0 \leq x \leq 1$, which implies (1).

To prove (2), let $g(x) := \frac{2}{3} \sqrt{4 + 3 x} + \frac{2-x}{3}$. Then one checks $g(0) = 2$ and $g'(x) = \frac{1}{\sqrt{4 + 3 x}} - \frac{1}{3} > 0$
for $0 \leq x \leq 1$, which implies $g(x) \geq 2$ for all $0 \leq x \leq 1$.
\end{proof}

\begin{proof}[Proof of Theorem \ref{T:ineq}]
By replacing $x_i$ by $1-x_i$ and setting $f(x) := \frac{2}{\pi} \arccos(1-x)$, our claim 
is equivalent to 
$$\sum_{i = 1}^m f(x_i + x_{i+1} - x_i x_{i+1}) \geq 2$$
under the constraint $\sum_{i=1}^m x_i = 1$, 
with $m \geq 3$ and $0 \leq x_i \leq 1$.

Let us set $\Delta_i := x_i + x_{i+1} - x_i x_{i+1}$ and $\sigma := \sum_{i =1}^m x_i x_{i+1}$.
Observe that $2 \sigma \leq (\sum_{i=1}^m  x_i)^2 = 1$.
Then we have by Lemma \ref{T:calc}
\begin{align*}
\sum_{i = 1}^m f(\Delta_i) & \geq  \frac{2}{3} \sum_{i=1}^m  \sqrt{\Delta_i} + \frac{1}{3} \sum_{i=1}^m \Delta_i \\
\intertext{and using Lemma \ref{T:sqrt} and the fact $\sum_{i=1}^m \Delta_i = 2 - \sigma$, we obtain}
&\geq  \frac{2}{3} \sqrt{4 + 3 \sigma} + \frac{1}{3} (2- \sigma)  \geq 2
\end{align*}
where in the last step we apply Lemma \ref{T:calc} (2). This completes the proof of the inequality.
By Lemma \ref{T:calc} (1), equality implies that $\Delta_i = 0, 1$ for every $i$, which in turn implies that $x_i = 0, 1$ 
for all $i$. Since $\sum_{i = 1}^m x_i = 1$, this can only happen if $x_i = 1$ for exactly one index $i$.
\end{proof}

\section{The obtuse angle case} \label{S:obtuse}

The proof in the previous section works as long as all angles $\gamma_i$ are less or equal than $\pi/2$. If one of them is obtuse,  
%(note that only one of them may be so, since the sum satisfies $\sum_{i=1}^m \gamma_i = \pi$), 
we have a geometric argument to reduce ourselves to that case. 

\subsection{Neutralizing pairs}
We call a  \emph{neutralizing pair} for $P$ a pair $\{\gamma_i, \gamma_{i+1}\}$ of adjacent interior angles of $P$ with $\gamma_i + \gamma_{i+1} \leq \pi$. 
Whenever we have a neutralizing pair, we can apply the following lemma. 

\begin{lemma} \label{L:pentagon}
Let $ABCDE$ be a hyperbolic pentagon, with right angles $\widehat{B}$ and $\widehat{E}$, and suppose that 
$\widehat{C} < \pi/2$ and $\widehat{C} + \widehat{D} \leq \pi$. Let $P$ be the midpoint of $\overline{CD}$, and let $\widehat{F}$ be the foot of the orthogonal projection of $P$ to $\overline{BC}$. 
Let $\widehat{G}$ be the intersection of the lines $\overline{FP}$ and $\overline{ED}$. Then the angle $\delta = D \widehat{G} F$ satisfies $\delta \leq \pi/2$.
\end{lemma}

\begin{proof}
Let $F'$ be the symmetric point to $F$ with respect to $P$. Then $CFP$ and $DPF'$ are equal triangles. 
Hence $E \widehat{D} F' = E \widehat{D} P + P \widehat{D} F' = E \widehat{D} C + B \widehat{C} D \leq \pi$, 
hence $F'$ lies on the segment $\overline{P G}$. Moreover, $D \widehat{F'} P = C \widehat{F} P = \pi/2$, 
hence $\delta = D \widehat{G} P \leq \pi/2$.
\end{proof}

\begin{center}
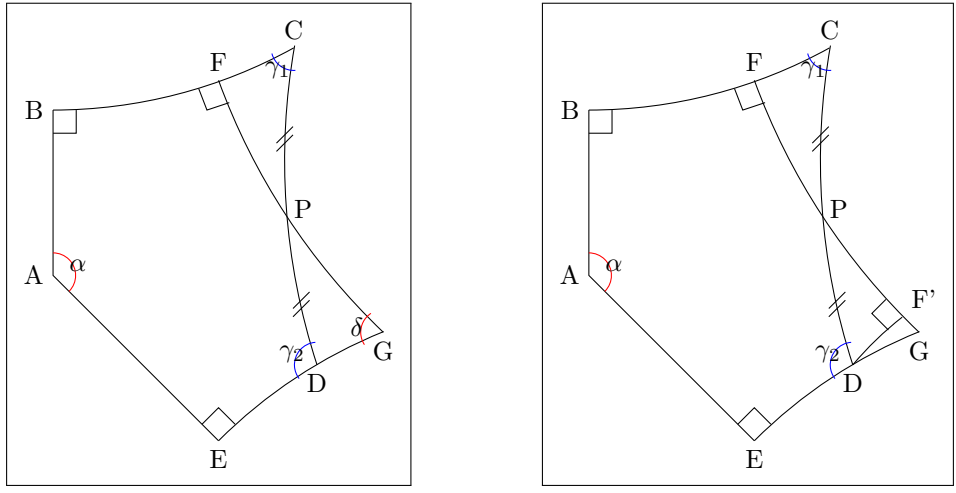
\begin{figure}
\fbox{\begin{tikzpicture}[scale = 2.2]
\draw (1, -1) coordinate (e) node[below] {E} --  (0,0) coordinate (a) node[left] {A}  -- (0,1) coordinate (b) node[left] {B} ;
\draw (0, 1) arc (270: 299 : 3 cm) coordinate (c) node[above] {C} ;
\draw (1, -1) arc (135 : 112 : 3 cm) ;
\draw (1, 1.18) coordinate (f) node[above] {F}  arc (200 : 226.3 : 4 cm) coordinate (g) node[below] {G};
\draw (1.46, 1.39) arc (170 : 198 : 4 cm) coordinate (d) node[below] {D};
\draw pic["$\alpha$", draw=red, -, angle eccentricity=1.2, angle radius=0.3 cm]
    {angle=e--a--b};
\draw pic["$\gamma_1$", draw=blue, -, angle eccentricity=1.2, angle radius=0.3 cm]
    {angle=b--c--d};
\draw pic["$\gamma_2$", draw=blue, -, angle eccentricity=1.2, angle radius=0.3 cm]
    {angle=c--d--e};
\draw pic["$\delta$", draw= red, -, angle eccentricity=1.2, angle radius=0.3 cm]
    {angle=f--g--e};
\draw (0, 0.86) -- (0.14, 0.86) -- (0.14, 1); 
\draw (0.9, - 0.9) -- (1.0, - 0.8) -- (1.1, -0.9); 
\draw (0.88, 1.13) -- (0.93, 1.0) -- (1.065, 1.045); 
\draw (1.4, 0.4)  coordinate (p) node[right] {P};
\draw (1.35, 0.8) -- (1.45, 0.9);
\draw (1.35, 0.75) -- (1.45, 0.85);
\draw (1.45, -0.2) -- (1.55, -0.1);
\draw (1.45, -0.25) -- (1.55, -0.15);
\end{tikzpicture}}
\hspace{1.5 cm}
\fbox{\begin{tikzpicture}[scale = 2.2]
\draw (1, -1) coordinate (e) node[below] {E} --  (0,0) coordinate (a) node[left] {A}  -- (0,1) coordinate (b) node[left] {B} ;
\draw (0, 1) arc (270: 299 : 3 cm) coordinate (c) node[above] {C} ;
\draw (1, -1) arc (135 : 112 : 3 cm) ;
\draw (1, 1.18) coordinate (f) node[above] {F}  arc (200 : 226.3 : 4 cm) coordinate (g) node[below] {G};
\draw (1.46, 1.39) arc (170 : 198 : 4 cm) coordinate (d) node[below] {D};
\draw pic["$\alpha$", draw=red, -, angle eccentricity=1.2, angle radius=0.3 cm]
    {angle=e--a--b};
\draw pic["$\gamma_1$", draw=blue, -, angle eccentricity=1.2, angle radius=0.3 cm]
    {angle=b--c--d};
\draw pic["$\gamma_2$", draw=blue, -, angle eccentricity=1.2, angle radius=0.3 cm]
    {angle=c--d--e};
\draw (0, 0.86) -- (0.14, 0.86) -- (0.14, 1); 
\draw (0.9, - 0.9) -- (1.0, - 0.8) -- (1.1, -0.9); 
\draw (0.88, 1.13) -- (0.93, 1.0) -- (1.065, 1.045); 
\draw (1.4, 0.4)  coordinate (p) node[right] {P};
\draw (d) arc (140 : 128.1 : 2 cm) coordinate (f') node[above right] {F'};
\draw (1.35, 0.8) -- (1.45, 0.9);
\draw (1.35, 0.75) -- (1.45, 0.85);
\draw (1.45, -0.2) -- (1.55, -0.1);
\draw (1.45, -0.25) -- (1.55, -0.15);
\draw (1.81 ,-0.33) -- (1.71, -0.23) -- (1.795,-0.145);
\end{tikzpicture}}
\caption{The hyperbolic pentagon of Lemma \ref{L:pentagon}. }
\end{figure}
\end{center}

We say that $P$ has \emph{disjoint neutralizing pairs} if every obtuse angle of $P$ belongs to a neutralizing pair, and all such neutralizing pairs are disjoint. Let us use the notation 
$$\varphi(x_1, x_2, \dots, x_m) := \sum_{i = 1}^m \frac{1}{1 + e^{2 x_i}}.$$

\begin{proposition} \label{P:reduce}
Let $P$ be a centrally symmetric hyperbolic polygon with $2m$ sides and center $o$, and let $\ell_1, \dots, \ell_m$ be the distances between $o$ 
and the midpoints of the sides. If $P$ has disjoint neutralizing pairs, there exists a centrally symmetric hyperbolic $2m$-gon $P'$ 
with no obtuse angles and such that 
$$\varphi(\ell_1, \ell_2, \dots, \ell_m) \leq \varphi(d'_1, d'_2, \dots, d'_m) $$
where $d_i'$ is the distance between $o$ and the $i$th side of $P'$.
\end{proposition}

\begin{proof}
Let us denote as $d_i$ the distance between $o$ and the $i$th side of $P$. Note that by definition $d_i \leq \ell_i$ for all $i$.

If the polygon $P$ only has acute angles, we take $P = P'$ and note that by definition $d_i' = d_i \leq \ell_i$, which yields the claim. 

Suppose now that the hyperbolic polygon $P$ has one obtuse angle, say $\gamma_1$, which belongs to a neutralizing pair, and let $\ell_1$ correspond to the side adjacent to the obtuse angle and the other angle, say $\gamma_2$, in the neutralizing pair.
Consistently with this choice, let us denote as $s_1, s_2, \dots, s_{2m}$ the sides of $P$.

Let us now consider the hyperbolic pentagon delimited by $s_{2m}, s_1, s_2$, and the orthogonal projections from $o$ to $s_2$ and $s_{2m}$. 
Let us call this pentagon $ABCDE$, where $o = A$, the side $s_1$ is denoted $\overline{DC}$, the orthogonal projection from $o$ to $s_2$ is $B$, 
and the orthogonal projection from $o$ to $s_{2m}$ is $E$.

Using Lemma \ref{L:pentagon}, let us replace $P$ by a new polygon $P'$ obtained substituting the pentagon $ABCDE$ by the pentagon $ABFGE$, which satisfies $\widehat{F} = \pi/2 $ and $\widehat{G} \leq \pi/2$.
If we denote by $d_1'$ the distance between $o = A$ and $\overline{FG}$, then we have 
$$d_1' = d(A, \overline{FG}) \leq d(A, P) = \ell_1.$$ 
On the other hand, note that for $i = 2, \dots, m$ the distance between $o$ and the $i$th side is the same for $P$ and $P'$. That is, 
$d_i = d_i'$ for $i = 2, \dots, m$.
Hence, 
$$\varphi(\ell_1, \ell_2, \dots, \ell_m) \leq \varphi(\ell_1, d_2, \dots, d_m) \leq \varphi(d_1', d_2, \dots, d_m) = \varphi(d_1', d_2', \dots, d_m').$$
If there are more than one neutralizing pairs, we can analogously replace each side adjacent to the pair by rotating it around its midpoint.
This proves the claim.
\end{proof}

%This deals with polygons with one obtuse angle, completing the proof of Theorem \ref{T:ineq-mcs}. 

\subsection{The general case}

%We are left to discuss the special case when $m$ is odd. Then, there are two elliptic cycles, 
%and the total angle around each cycle is $2 \pi$. Thus, we are given a centrally symmetric polygon $P$
%whose sum of angles is $4 \pi$. If all angles of $P$ are acute, the previous proof works. 

Let $(p_i)_{i =1}^{2m}$ denote the vertices of $P$ and $(q_i)_{i =1}^{2m}$ denote the midpoints of the sides, indexed so that $q_i$ lies between $p_{i-1}$ and $p_{i}$. Let $o$ denote the center of symmetry of $P$. 
Let $\alpha_i = q_i \widehat{o} q_{i+1}$ be the angles at the origin,  and $\gamma_i = q_i \widehat{p}_i q_{i+1}$ the angles at the vertices of $P$. By the cycle condition and symmetry we have 
$$\sum_{i = 1}^m \alpha_i = \pi, \qquad \sum_{i = 1}^{m} \gamma_i = \frac{2 \pi}{k},$$
where $k \geq 1$ is an integer. Note that if $k \geq 2$, at most one of the $\gamma_i$ is obtuse, hence $P$ has disjoint neutralizing pairs. 
However, if $k = 1$, $P$ need not have disjoint neutralizing pairs; in particular, it may have three consecutive obtuse angles. In order to deal with this case, we need the notion of \emph{dual polygon}. 

\begin{figure}
\includegraphics[width=0.48 \textwidth]{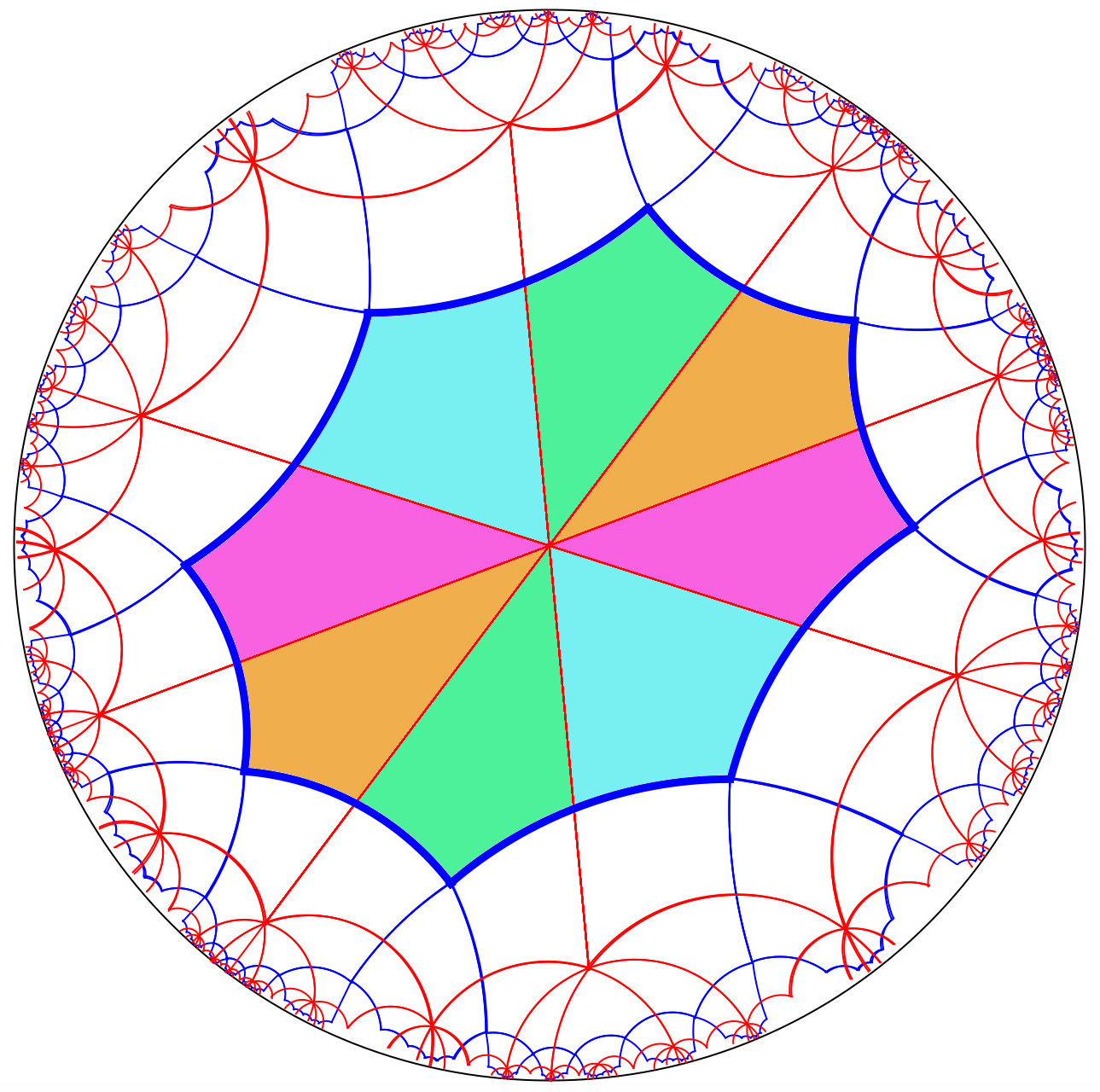}
\includegraphics[width=0.48 \textwidth]{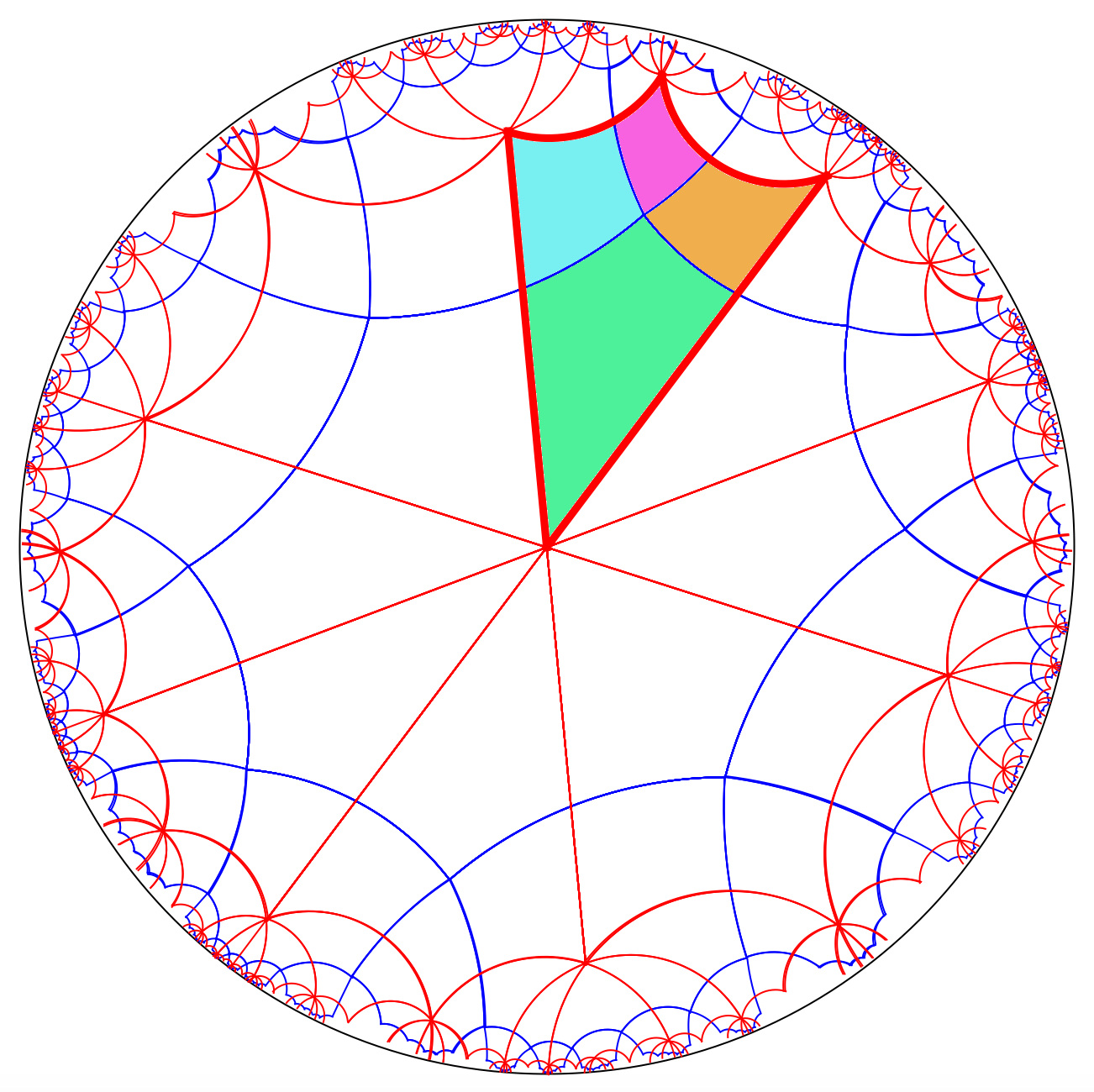}
\caption{On the left: the polygon $P$, in blue. On the right: the dual polygon $\widehat{P}$, in red. The highlighted quadrilaterals can be rearranged as shown to form the dual polygon.}
\label{F:dual}
\end{figure}

\subsection{Dual polygons}

Given a centrally symmetric polygon $P$ with center $o$, we construct its \emph{dual polygon} $\widehat{P}$
as follows. 

%Let $\ell_1, \dots, \ell_{2m}$ be the length of the segments between $o$ and the midpoint $P_i$ of the $ith$ side, 
%and moreover let $(A_i)_{i =1}^{2m}$ denote the vertices of $P$ so that $P_i$ lies between $A_i$ and $A_{i+1}$. 
%Let $Q_i$ be the quadrilateral delimited by $O, P_i, A_i, P_{i+1}$. 

Let $Q_i$ be the quadrilateral delimited by $o, q_i, p_i, q_{i+1}$. 
As in Figure \ref{F:dual}, we can cut and rearrange the $Q_i$'s with $1 \leq i \leq m$ by gluing all vertices $p_i$ to a single point, which we now denote as $v$. Since the sum of all angles at $p_i$ is $2 \pi$, this creates a new polygon with sides of lengths $2 \ell_1, \dots, 2 \ell_m$. The angles of $\widehat{P}$ are $\alpha_1, \dots, \alpha_m$, hence their sum is $\pi$. 
We define the pair $(\widehat{P}, v)$ to be the dual polygon to $(P, o)$. 

The duality relation 
$$(P, o) \leftrightarrow (\widehat{P}, v)$$ 
defines a bijective correspondence between centrally symmetric $2m$-gons with sum of angles $4 \pi$ 
and $m$-gons with sum of angles $\pi$ together with a choice of a point inside them. 

Given a polygon $P$ with $2m$ sides and a point $o$ inside $P$, we define 
$$\Sigma(P) := \sum_{i =1}^{2 m} \frac{1}{1 + e^{2 \ell_i}} $$
where $\ell_i$ are the segments connecting $o$ and the midpoint of the $ith$ side.
Let us also define 
$$\widehat{\Sigma}(P) := \sum_{i = 1}^m \frac{1}{1 + e^{s_i}}$$ 
where $s_i$ are the lengths of the sides of $P$. 
Then note that we have 
$$\Sigma(P) = \widehat{\Sigma}(\widehat{P}).$$ 
In particular, $\Sigma(P)$ does \emph{not} depend on $v$ but only on 
$\widehat{P}$.

\begin{lemma} \label{L:dual-trick}
Let $P$ be a centrally symmetric hyperbolic polygon with $2m$ sides and total sum of its interior angles $4 \pi$. 
Then there exists a centrally symmetric hyperbolic polygon $P'$ with the same number of sides, 
so that $\Sigma(P) = \Sigma(P')$ and so that $P'$ has at most four obtuse angles, which belong to disjoint neutralizing pairs. 
\end{lemma}

\begin{proof}
Let $\widehat{P}$ be the dual polygon to $P$, as defined above. 
We claim that we can pick another point $v'$ inside $\widehat{P}$ so that at most two of the angles 
at $v'$ are obtuse. This is just because we can pick two non-adjacent sides of  $\widehat{P}$ and join their midpoints 
by a segment. Now, let us pick $v'$ on that segment and connect it to all midpoints of the sides of $\widehat{P}$. 

Then, out of the angles $\gamma_i' := q_i \widehat{v'} q_{i+1}$ with $1 \leq i \leq m$, at most two of them can be obtuse. 
Then we define $P'$ to be the dual of $(\widehat{P}, v')$. Since $P$ and $P'$ have the same dual, we have 
$\Sigma(P) = \Sigma(P')$.
Thus, in $P'$ there are at most $4$ obtuse angles $\gamma_i'$, and for all of them there exists another adjacent angle $\gamma'_{i\pm 1}$ 
so that $\gamma'_i + \gamma'_{i\pm 1} < \pi$.  Hence, $P'$ has neutralizing pairs.
\end{proof}

\medskip

By putting together these reductions we can complete the proof of Theorem \ref{T:ineq-mcs}. 
Let us see the details.

\begin{proof}[Proof of Theorem \ref{T:ineq-mcs}]
Let us first suppose that $\gamma_i \leq \pi/2$ for all $i$. 
We know by \eqref{E:arccos} that $\sum_{i = 1}^m \arccos(z_i z_{i+1}) \leq \pi$ with $0 < z_i < 1$. Then we need to show that $\sum_{i = 1}^m z_i > m- 1$.
Suppose not, then there exists $z_i$ with $\sum_{i = 1}^m z_i \leq m-1$. Then there exists $(z_i')_{i = 1}^m$ with $0 \leq z_i \leq z_i' \leq 1$ 
for all $i$, so that $\sum z_i' = m - 1$. Then we have, by Theorem \ref{T:ineq},
$\pi \leq \sum_{i = 1}^m \arccos ( z_i' z_{i+1}') \leq  \sum_{i = 1}^m \arccos ( z_i z_{i+1}) \leq \pi$, 
hence $\sum_{i = 1}^m \arccos ( z_i' z_{i+1}') = \pi$, which by the second part of Theorem \ref{T:ineq} implies $z_i' = 0$ for some $i$, hence 
also $z_i = 0$, which is a contradiction. 

In the general case, we first apply Lemma \ref{L:dual-trick} to reduce to the case where $P$ has disjoint neutralizing pairs. 
Then, by applying Proposition \ref{P:reduce}, we reduce to the case of $P$ having no obtuse angles, 
which we can deal with as above. This completes the proof. 
\end{proof}

\begin{proof}[Proof of Theorem \ref{T:main}]
Theorem \ref{T:ineq-mcs} shows that the criterion of Theorem \ref{T:criterion} holds, proving the singularity of hitting measure.
\end{proof}

%\textbf{Remark. } This argument shows that \eqref{E:McS3} holds whenever all interior angles of $P$ are not obtuse. Requiring that the sum of the interior angles is $2\pi$ is needed for two reasons only:
%\begin{enumerate}
%	\item the group generated by the hyperbolic translations needs to act geometrically due to Poincar\'e's theorem on fundamental polygons;
%	\item it allows us to apply Lemma \ref{L:pentagon} in case one of the interior angles turned out to be obtuse.
%\end{enumerate}

\section{Coxeter groups}

Let $P$ be a centrally symmetric convex polygon with $2 m$ sides in $\mathbb{H}^2$, with each angle $\gamma_i$ at the vertices being equal to $\frac{\pi}{k_i}$ for some natural $k_i > 1$, for $1 \leq i \leq 2m$. Then, due to \cite[Theorem 6.4.3]{Davis}, the group of isometries generated by hyperbolic reflections $R := \{ r_1, \dots, r_{2 m} \}$ with respect to the sides of $P$ acts geometrically on $\mathbb{H}^2$. Therefore, it is a hyperbolic group, so Theorem \ref{T:BHM} can be applied to it. Such groups are referred to as \textit{hyperbolic Coxeter groups}. 
%, and in this paper we will denote them by $\Gamma^r_P$ (where ``r'' stands for reflection).

Below we will show that Theorem \ref{T:main} can be quickly generalized to hyperbolic Coxeter groups.

\begin{lemma}
	\label{L:coxeter_cyl}
	Let $m > 1$. Consider a random walk on the free product of $2m$ copies of $\mathbb{Z} / 2 \mathbb{Z}$
	$$F'_{2m} = \left\langle s_1, \dots, s_{2m} \ | \ s_i^2 = e \right\rangle,$$ 
	defined by a probability measure $\mu$ on the generators. If we denote  $x_i := F_\mu(e, s_i)$ for $1 \le i \le 2m$, and the hitting measure on the boundary of $F'_{2m}$ by $\nu$, then 
	\[
	\nu(C(s_i)) = \frac{x_i}{1+x_i}.
	\]
\end{lemma}

\begin{proof}
The proof of this lemma can be obtained in a similar way to the proof of Lemma \ref{L:cylinder} for $F_m$, because the Cayley graphs for $F_m$ and $F'_{2m}$ are  isometric. % as metric spaces.
%For any infinite word $w = s_{j_1} s_{j_2} s_{j_3} \dots$ there exist two possibilities:
%	\begin{enumerate}
%		\item There exists a subword $s_{j_1} \dots s_{j_k}$ such that it equals $s_i$ in $F_m$
%		\item No subword $s_{j_1} \dots s_{j_k}$ equals $s_i$, so it belongs to the set of paths which never hit $s_i$.
%	\end{enumerate}
%	In the first case we denote this subword by $w_1$, and we consider $w_1^{-1} w$ and we apply the same procedure, but replacing $s_i$ with $s_i^{-1}$ at each subsequent step. This procedure yields the equality

More precisely, a sample path converges to the boundary of the cylinder $C(s_i)$ if and only if 
it crosses the edge $s_i$ an odd number of times. This leads to the following computation:
	\[
	\begin{aligned}
		\nu(C(s_i)) &= \mathbb{P}(e \rightarrow s_i \nrightarrow e) + \mathbb{P}(e \rightarrow s_i \rightarrow e \rightarrow s_i \nrightarrow e) + \dots = \\ 
		&= \sum_{n=0}^{\infty} F_\mu(e, s_i)^{2 n+1} (1 - F_\mu(e, s_i)) \\
		& = \sum_{k = 1}^\infty (-1)^{k+1} x_i^k = \frac{x_i}{1 + x_i}.
		%F_\mu(e, s_i)(1 - F_\mu(e, s_i^{-1})) \sum_{n=0}^{\infty} \left(F_\mu(e, s_i)F_\mu(e, s_i^{-1})\right)^n = \frac{x_i (1 - \check{x}_{i})}{1 - x_i \check{x}_{i} }.
	\end{aligned}
	\]

\end{proof}

%\textbf{Remark.} As $s_i^2 = e$ for all $1 \le i \le 2m$, the usual notion of a symmetric measure loses any meaning in this context. 
%Measures satisfying the conditions of this theorem can be thought of as \textit{geometrically symmetric}.

A measure $\mu$ on the set $R = \{ r_1, \dots, r_{2 m} \}$ of reflections through the sides of $P$ is called 
 \textit{geometrically symmetric} if $\mu(r_i) = \mu(r_{i+m})$ for each $1 \leq i \leq m$. 

\begin{theorem}
	\label{T:Coxeter_ineq}
	Let $\mu$ denote a geometrically symmetric measure supported on the generators $R = \{ r_1, \dots, r_{2 m} \}$ of a hyperbolic Coxeter group. Suppose that
	\begin{equation}
		\label{coxeter_ineq_itself}
		\sum_{i=1}^m \dfrac{1}{1 + e^{\ell(r_i r_{i+m})/2}} < \dfrac{1}{2}.
	\end{equation}
	Then the hitting measure $\nu$ in $\partial \mathbb{D}$ is singular with respect to Lebesgue measure. 
\end{theorem}
\begin{proof}
	The proof of this theorem is quite similar to the proof of Theorem \ref{T:criterion}. We consider a measure $\tilde{\mu}$ on a free product $\left\langle h_1, \dots, h_{2m} \ | \ h_i^2 = e \right\rangle $ of $2m$ copies of $\mathbb{Z} / 2\mathbb{Z}$ uniquely defined by $\tilde{\mu}(h_i) = \mu(r_i)$.
	
	If $\nu$ were to be absolutely continuous, then a similar argument would yield that
	\[
	\begin{aligned}
		\ell(r_i r_{i+m}) &\le d_{\mu}(e, r_i r_{i+m}) \le d_{\mu}(e, r_i) + d_{\mu}(e, r_{i+m})  \\ &\le d_{\tilde{\mu}}(e, h_i) + d_{\tilde{\mu}}(e, h_{i+m}) = 2 d_{\tilde{\mu}}(e, h_i) = -2\log \, x_i.
	\end{aligned}
	\]
	Keep in mind that $d_{\tilde{\mu}}(e, h_i) = d_{\tilde{\mu}}(e, h_{i+m})$ due to $\tilde{\mu}$ being geometrically symmetric as well.
	Therefore, 
	\[
	\dfrac{x_i}{1+x_i} \le \dfrac{1}{1 + e^{\ell(r_i r_{i+m})/2}}
	\]
	and due to Lemma \ref{L:coxeter_cyl} we obtain
	\[
	1 = \sum_{i=1}^{2m} \frac{x_i}{1 + x_i} \leq 2 \sum_{i=1}^{m} \dfrac{1}{1 + e^{\ell(r_i r_{i+m})/2}} < 1, 
	\]
	which delivers a contradiction.
\end{proof}

\begin{theorem}
	\label{T:Coxeter-sing}
	The hitting measure of a nearest-neighbour random walk generated by a geometrically symmetric measure on a Coxeter group associated with a centrally symmetric polygon is singular with respect to Lebesgue measure on $\partial \mathbb{D}$.
\end{theorem}
\begin{proof}
	Let us recall that $(g_i)^m_{i =1}$ denotes the translations identifying the opposite sides of $P$. It is easily seen that $\ell(r_i r_{i+m}) = 2\ell(g_i) = 2\ell(g_{i+m})$ for every $1 \le i \le m$. However, we can apply Theorem \ref{T:ineq-mcs} because there are no obtuse angles, to get
	\[
	\sum_{i = 1}^m \frac{2}{1 + e^{\ell(r_i r_{i+m})/2}} = \sum_{g \in S} \frac{1}{1 + e^{\ell(g)}} < 1.
	\]

	We finish the proof by applying Theorem \ref{T:Coxeter_ineq}.
\end{proof}

\end{document}